\documentclass[12pt]{article}
\usepackage{amsmath, amsfonts, amsthm,amssymb, color, enumerate, extarrows, xfrac, indentfirst, setspace, hyperref}

\usepackage{enumitem}
\setlist[description]{leftmargin=2.62cm, labelindent=1cm}
\newtheorem{theorem}{Theorem}
\newtheorem{definition}[theorem]{Definition}
\newtheorem{cor}[theorem]{Corollary}

\DeclareMathOperator{\Imagine}{Im}

\usepackage{cite}

\hypersetup{
	colorlinks   = true,
	citecolor    = magenta
}

\numberwithin{equation}{section}

\usepackage[titletoc]{appendix}

\def\CC{{\rm\kern.24em\vrule
width.02em height1.4ex
depth-.05ex\kern-.26em C}}

\def\1bar{\overline{1}}
\def\2bar{\overline{2}}

\begin{document}
\title{\bf Rigidity theorem by the minimal point of the Bergman kernel \rm}
\author{Robert Xin Dong\quad and\quad John Treuer}
\date{}
\maketitle

\renewcommand{\thefootnote}{\fnsymbol{footnote}}
\footnotetext{\hspace*{-7mm} 
\begin{tabular}{@{}r@{}p{16.5cm}@{}}
& 2010 Mathematics Subject Classification. Primary 30C40; Secondary 30C35, 30C85, 30C20\\
& Key words and phrases.
Bergman Kernel, Minimal Domain, Suita Conjecture, Szeg\"o Kernel

\end{tabular}}

\begin{abstract}
We use the Suita conjecture (now a theorem) to prove that for any domain $\Omega \subset \mathbb{C}$ its Bergman kernel $K(\cdot, \cdot)$ satisfies $K(z_0, z_0) = \hbox{Volume}(\Omega)^{-1}$ for some $z_0 \in \Omega$ if and only if $\Omega$ is either a disk minus a (possibly empty) closed polar set or $\mathbb{C}$ minus a (possibly empty) closed polar set. When $\Omega$ is bounded with $C^{\infty}$-boundary, we provide a simple proof of this using the zero set of the Szeg\"o kernel.  Finally, we show that this theorem fails to hold in $\mathbb{C}^n$ for $n > 1$ by constructing a bounded complete Reinhardt domain (with algebraic boundary) which is strongly convex and not biholomorphic to the unit ball $\mathbb{B}^n \subset \mathbb{C}^n$. 
 
\end{abstract}

\section{Introduction}
\indent
Let $\Omega$ be a domain in $\mathbb{C}^n$ and denote the Bergman space of $\Omega$ by $A^2(\Omega) = L^2(\Omega) \cap \mathcal{O}(\Omega)$ where $\mathcal{O}(\Omega)$ is the set of holomorphic functions on $\Omega$.  The Bergman space is a separable Hilbert space under the $L^2$-inner product with Lebesgue volume measure.  If $\{\phi_j\}_{j=0}^{\infty}$ is a complete orthonormal basis of $A^2(\Omega)$, then the Bergman kernel function $K: \Omega \times \Omega \to \mathbb{C}$ defined by
\begin{equation} \label{onb expansion}
K(z, w) = \sum_{j=0}^{\infty} \phi_j(z)\overline{\phi_j(w)}
\end{equation}
is the unique function on $\Omega \times \Omega$ which satisfies the properties
\begin{enumerate}
\item For all $w \in \Omega$, $K(\cdot, w) \in A^2(\Omega)$
\item $K(z, w) = \overline{K(w, z)}$
\item For all $f \in A^2(\Omega)$ and $z \in \Omega$,
\begin{equation} \label{reproducing}
f(z) = \int_{\Omega} f(w)K(z, w) {dv (w)}.
\end{equation} 
\end{enumerate}
We note that the definition of $K(z, w)$ is independent of the particular orthonormal basis chosen.  We shall use the notation $K_{\Omega}$ for the Bergman kernel of $\Omega$ instead of $K$ when we wish to emphasize that $\Omega$ is the domain under consideration.  If $f: \Omega_1 \to \Omega_2$ is a biholomorphic map, then the Bergman kernels for the respective domains are related by the transformation law of the Bergman kernel:
$$
K_{\Omega_1}(z, w) = f'(z)K_{\Omega_2}(f(z), f(w))\overline{f'(w)}.
$$  For further background on the Bergman kernel, we refer the readers to Krantz's book \cite{K01}.
\medskip

If $\Omega$ is a bounded domain in $\mathbb{C}^n$ and $v$ is the Lebesgue $\mathbb{R}^{2n}$-measure, then $ v(\Omega)^{-\sfrac{1}{2}} \in A^2(\Omega)$ and $\|v(\Omega)^{-\sfrac{1}{2}}\|_{L^2} = 1$.  Hence, if $\{v(\Omega)^{-\sfrac{1}{2}}\} \cup \{\phi_j\}_{j=1}^{\infty}$ is a complete orthonormal basis for $A^2(\Omega)$, then by \eqref{onb expansion},

\begin{equation} \label{volume inequality}
K(z, z) \geq {1 \over v(\Omega)}, \quad z \in \Omega.
\end{equation}

\noindent{}Equality in \eqref{volume inequality} is achieved for the unit ball $\mathbb{B}^n \subset \mathbb{C}^n$, $n \geq 1$ with $z = 0$ because $K_{\mathbb{B}^n}(z, w) = v(\mathbb{B}^n)^{-1}(1 - z\overline{w})^{-n - 1}$.  
\medskip

In this paper, we completely classify the domains in $\mathbb{C}$ for which equality in \eqref{volume inequality} holds at some (minimal) point in the domain.

\begin{theorem} \label{theorem 1}
Let $\Omega \subset \mathbb{C}$ be a domain.  Suppose there exists a $z_0 \in \Omega$ such that
\begin{equation} \label{=}
K(z_0, z_0) = {1 \over v(\Omega)},
\end{equation}
where we use the convention $v(\Omega)^{-1} = 0$ if $v(\Omega) = \infty$. 

\begin{enumerate}[label=(\roman*)]

\item \label{Case 1} If $v(\Omega) = \infty$, then $\Omega = \mathbb{C} \setminus P$ where $P$ is a possibly empty, closed polar set.

\item \label{Case 2} If $v(\Omega) < \infty$, then $\Omega = D(z_0, r) \setminus P$ where $P$ is a possibly empty, polar set closed in the relative topology of $D(z_0, r)$ with $r = \sqrt{v(\Omega)\pi^{-1}}$.

\end{enumerate}
\end{theorem}

 \noindent{}We remark that
\begin{enumerate}

\item \label{Converse} A set $P$ is said to be polar if there is a subharmonic function $u \not \equiv -\infty$ on $\mathbb C$ such that $P  \subset \{z \in \mathbb C\,: \, u(z)=-\infty\}$. If $P$ is a closed polar subset of a domain $\Omega$, then $A^2(\Omega \setminus P) = A^2(\Omega)$ by \cite{S82}.   It follows that if $\Omega_1 = D(z_0, r) \setminus P_1$ and $\Omega_2 = \mathbb{C} \setminus P_2$ where $P_i$, $i = 1, 2$, are relatively closed polar sets, then $K_{\Omega_i}$ extends to $\Omega_i \times \Omega_i$ and
$$
K_{\Omega_1}(z_0, z_0) = {1 \over v(\Omega_1)}, \quad K_{\Omega_2}(0, 0) = 0.
$$

\item Compact subsets of polar sets are totally disconnected.  So the polar set $P$ will be empty if, for instance, the boundary of $\Omega$ is parametrized by non-trivial simple closed curves.

\item We can see that in Theorem \ref{theorem 1}, Statement (i) still holds in the Riemann surface setting, since on any non-planar open Riemann surface the Bergman kernel does not vanish. However, Statement (ii) cannot be generalized to an arbitrary open Riemann surface. 

For example, let $X_{\tau, u}:=X_\tau \backslash \{u\}$ be an open Riemann surface obtained by removing one single point $u$ from a compact complex torus $X_\tau:=\mathbb C/\left(\mathbb Z+\tau\mathbb Z\right)$, for $\tau\in \mathbb C$ and $\Imagine \tau>0$. By the removable singularity theorem, the Bergman kernel on $X_{\tau, u}$ is the 2-form $K_{\tau, u} =(\Imagine \tau)^{-1} dz\wedge d\bar z,$ where $z$ is the local coordinate induced from the complex plane $\mathbb C$ (see \cite{D14, D17}). Here $v(X_{\tau, u})$ is precisely $\Imagine \tau$, the area of the fundamental parallelogram. So \eqref{=} holds true for $X_{\tau, u}$, which is not biholomorphic to $D(z_0, r) \setminus P$.

\item A bounded domain $\Omega \subset \mathbb{C}^n$ is called a minimal domain with a center $z_0 \in \Omega$ if $
v(\Omega) \leq v(\Omega')
$ for any biholomorphism $\varphi: \Omega \to \Omega'$ such that $\det(\hbox{J}\varphi(z_0)) = 1$, where $\hbox{J}\varphi$ denotes the Jacobian matrix of $\varphi$.  It is known \cite{M56} that equivalently a domain $\Omega$ is minimal with the center at $z_0$ if and only if 
$$
K(z, z_0) \equiv  {1 \over v(\Omega)}, \quad z \in \Omega.
$$

For more information about minimal and representative domains, see \cite{M56, K63, IK, YZ}.  Theorem 1 classifies all minimal domains in $\mathbb{C}$.  More precisely, we have the following corollary.

\end{enumerate}

\begin{cor} \label{Corollary}
Let $\Omega \subset \mathbb{C}$ be a domain.  Suppose there exists a $z_0 \in \Omega$ such that 
$$
K(z, z_0)\equiv C, \quad \text{for any } z \in \Omega.
$$   
Then $C = v(\Omega)^{-1}$ and
\begin{enumerate} [label=(\roman*)]
\item If $C = 0$, then $\Omega = \mathbb{C}\setminus P$ where $P$ is a closed polar set.
\item If $C > 0$, then $\Omega = D(z_0, r)\setminus P$ where $P$ is a possibly empty, polar set closed in the relative topology of $D(z_0, r)$ with $r = \sqrt{v(\Omega)\pi^{-1}}$. 
\end{enumerate}

Consequently, all minimal domains with the center $z_0$ in $\mathbb{C}$ are disks centered at $z_0$ possibly minus closed polar sets.

\end{cor}

In $\mathbb{C}^n$, $n \geq 1$, Equality \eqref{=} is achieved for the unit ball and more generally for complete Reinhardt domains.  A domain $\Omega$ in $\mathbb{C}^n$ is said to be complete Reinhardt if for all $z = (z_1,..., z_n) \in \Omega$
$$
 (\lambda_1z_1,..., \lambda_nz_n) \in \Omega, \quad |\lambda_i| \leq 1.
$$
For a bounded complete Reinhardt domain, $\{z^{\alpha}\}_{\alpha \in \mathbb{N}^n}$ is a complete orthogonal system of $A^2(\Omega)$.  It follows from \eqref{onb expansion} that $K_{\Omega}(0, 0) = v(\Omega)^{-1}$.  Equality \eqref{=} also holds for complete circular domains (cf. \cite{Bo00}).
\medskip

When $n > 1$, the unit polydisk $D(0, 1)^n$ is an example of a complete Reinhardt domain which is not biholomorphic to $\mathbb{B}^n$.  Thus, to generalize Theorem \ref{theorem 1} to $\mathbb{C}^n$, $n \geq 1$, $D(z_0, r)$ cannot simply be replaced by a translation and rescaling of $\mathbb{B}^n$.  However, the polydisk does not have smooth boundary whereas the unit ball is strongly convex with algebraic boundary.  So, we also consider whether Theorem \ref{theorem 1} generalizes to $\mathbb{C}^n$ if $\Omega$ is required to be complete Reinhardt, strongly convex with algebraic boundary.  The answer is no as the next theorem shows.

\begin{theorem} \label{theorem 2}
Let $\Omega = \{z \in \mathbb{C}^2: |z_1|^4 + |z_1|^2 + |z_2|^2 < 1\}$ be a domain with algebraic boundary. Then $\Omega$ is complete Reinhardt, strongly convex and not biholomorphic to $\mathbb{B}^2$.  
\end{theorem}

The paper is organized as follows.  In Section 2, we give a proof of Theorem \ref{theorem 1} using only complex analysis of one variable for the case where $\Omega$ has $C^{\infty}$-boundary.  In Section 3, Theorem \ref{theorem 1} is proved in full generality using the Suita conjecture and Corollary \ref{Corollary} is proved.  In Section 4, Theorem \ref{theorem 2} is proved.

\section{Proof of $C^{\infty}$ boundary case of Theorem \ref{theorem 1}}

The main ingredient in the proof of Theorem \ref{theorem 1} will be \eqref{Bergman-Szego Relation}.  The theory that follows can be found in Bell's book \cite{B92}. We begin by recalling the Szeg\"o kernel.  
\medskip

Let $\Omega$ be a bounded domain with $C^{\infty}$-smooth boundary and denote its boundary by $b\Omega$.  Then $\Omega$ is $n$-connected with $n < \infty$ and the boundary consists of $n$ simple closed curves parametrized by $C^{\infty}$ functions $z_j: [0, 1] \to \mathbb{C}$.  Without loss of generality, let $z_n$ parametrize the outermost boundary curve; that is, $z_n$ parametrizes the boundary component which bounds the unbounded component of the complement of the domain. Additionally, the boundary component parametrized by $z_j$, $j = 1, ..., n$, is denoted by $b\Omega_j$.
\medskip
Let $T(z)$ denote the unit tangent vector to the boundary
and $ds$ denote the arc-length measure of the boundary.  Define $L^2(b\Omega) = \{f: b\Omega \to \mathbb{C}: \|f\|_{L^2(b\Omega)}< \infty \}$ where the norm $\| \cdot \|_{L^2(b\Omega)}$ is induced by the inner product 
$$
\langle f, g \rangle = \int_{b\Omega} f\bar{g}\ ds.
$$
Let $A^{\infty}(b\Omega)$ denote the boundary values of functions in $\mathcal{O}(\Omega) \cap C^{\infty}(\overline{\Omega})$.  The Hardy space of $b\Omega$ denoted $H^2(b\Omega)$ is the $L^2(b\Omega)$ closure of $A^{\infty}(b\Omega)$.  If $P: L^2(b\Omega) \to H^2(b\Omega)$ is the orthogonal projection, then the Szeg\"o kernel for $\Omega$, $S(z, a)$, is defined by
$$
P(C_a(\cdot) )(z) = S(z, a), \quad a, z \in \Omega, \quad C_a(z) = \overline{{1 \over 2\pi i}{T(z) \over z - a}}
$$
\cite[Section 7]{B92}.  It can be shown that $S(z, a) = \overline{S(a, z)}$, and from the proof of the Ahlfor's Mapping Theorem, for each $a$, $S(\cdot, a)$ has $n - 1$ zeros counting multiplicity \cite[Theorem 13.1]{B92}.  We note that the proof of the Ahlfor's Mapping Theorem just cited requires $C^{\infty}$-boundary regularity.  Since we will need the fact about the $n-1$ zeros of $S(\cdot, a)$, we have imposed a $C^{\infty}$ boundary regularity assumption on $\Omega$ in this section.
\medskip

Let $\omega_j$ be the (unique) solution to the Dirichlet boundary-value problem
\[
\begin{cases}
 \Delta u(z) = 0 & z \in \Omega \\
 u(z) = 1 & z \in b\Omega_j \\
 u(z) = 0 & z \in b\Omega_k, \quad  k \neq j
\end{cases}
\]
and define $F_j: \Omega \to \mathbb{C}$ by $F_j(z) = 2\partial \omega_j / \partial z$.  Then the Bergman kernel and Szeg\"o kernel are related by
\begin{equation} \label{Bergman-Szego Relation}
K(z, a) = 4\pi S(z, a)^2 + \sum_{j=1}^{n-1} \lambda_j F_j(z)
\end{equation}
where $\lambda_j$ are constants in $z$ and depend on $a$ \cite[Theorem 23.2]{B92}.  Since $\omega_j \in C^{\infty}(\overline{\Omega})$ is harmonic, $F_j \in \mathcal{O}(\Omega) \cap C^{\infty}(\overline{\Omega}) \subset A^2(\Omega)$.  We now prove Theorem \ref{theorem 1} when $\Omega$ is bounded with $C^{\infty}$ boundary.
\begin{proof}
After a translation we may assume that $z_0 = 0$.  Let $\{v(\Omega)^{-\sfrac{1}{2}}\} \cup \{\phi_j\}_{j=1}^{\infty}$ be a complete orthonormal basis for $A^2(\Omega)$. Then
$$
{1 \over v(\Omega)} = K(0, 0) = {1 \over v(\Omega)} + \sum_{j=1}^{\infty} \phi_j(0)\overline{\phi_j(0)},
$$
which implies that $\phi_j(0) = 0$, for all $j$.   It follows that $K(0, a) = v(\Omega)^{-1},$ and for any $f \in A^2(\Omega)$ by the reproducing property \eqref{reproducing},
$$
f(0)  = {1 \over v(\Omega)} \int_{\Omega} f(w) dv(w).
$$
In particular for $F_j, j = 1,..., n-1$,
\begin{eqnarray*}
F_j(0) &=& {1 \over 2i v(\Omega)} \int_{\Omega} 2{\partial \omega_j \over \partial w} d\bar{w} \wedge dw
\\
&=& {-1 \over i v(\Omega)} \int_{b\Omega} \omega_j d\bar{w}
\\
&=&
{-1 \over i v(\Omega)} \int_{b\Omega_j} 1 d\bar{w}
\\
&=&
0.
\end{eqnarray*}
Hence setting $z = 0$ in \eqref{Bergman-Szego Relation},
$$
{1 \over v(\Omega)} = K(0, a) = 4\pi S^2(0, a)
$$
Since $S(0, \cdot) = \overline{S(\cdot, 0)}$ has $n - 1$ zeros counting multiplicity, $n =1$; that is, $\Omega$ is simply-connected.
\medskip
Let $F:D(0, 1) \to \Omega$ be the inverse of the Riemann map with $F(0) = 0, F'(0) > 0$.  By the transformation law of the Bergman kernel,
$$
{1 \over \pi} = K_{D(0, 1)}(z, 0) = F'(z)K_{\Omega}(F(z), 0)\overline{F'(0)} = {F'(z)\overline{F'(0)} \over v(\Omega)}.
$$
So $F$ is linear; hence $\Omega = D(0, F'(0))$.   

\end{proof}

\section{Proof of Theorem \ref{theorem 1}}
The proof of Theorem \ref{theorem 1} in the previous section used methods specific to bounded domains with $C^{\infty}$-boundary and cannot be used to prove the general case where $\Omega \subset \mathbb{C}$ is a domain. So instead, we will use the Suita conjecture to prove the general case. 
\medskip

Let $SH(\Omega)$ be the set of subharmonic functions on $\Omega$. The (negative) Green's function of a domain $\Omega \subset \mathbb{C}$ is defined by
$$
g(z, w) = \sup\{ u(z):  u < 0, u \in SH(\Omega),\limsup_{\zeta \to w} (u(\zeta) - \log |\zeta - w|) < \infty \}.
$$
A domain admits a Green's function if and only if there exists a non-constant, negative subharmonic function on $\Omega$.  In particular any bounded domain has a Green's function.

\begin{definition}
The Robin constant is defined by
$$
\lambda(z_0) = \lim_{z \to z_0} g(z, z_0) - \ln|z - z_0|.
$$
\end{definition}

In 1972, Suita \cite{S72} made the following conjecture.
\newline
\newline
\noindent{}\textbf{Suita Conjecture. } If $\Omega$ is an open Riemann surface admitting a Green's function, then 
$$
\pi K(z, z) \geq e^{2\lambda(z)}
$$
and if equality holds at one point, then $\Omega$ is biholomorphic to $D(0, 1) \setminus P$, where $P$ is a possibly empty, closed (in the relative topology of $D(0, 1)$) polar set.
\bigskip

The inequality part of the Suita conjecture for planar domains was proved by B\l{}ocki \cite{Bl13}. See also \cite{BL16}. The equality part of the Suita conjecture was proved by Guan and Zhou \cite{GZ15}. See also \cite{D18} for related work.
\medskip

We will use the following result in Helm's book \cite[Theorem 5.6.1]{H14}, which is attributed to Myrberg \cite{M33}.

\begin{theorem} \label{Myrberg's Result}
Let $\Omega$ be a non-empty open subset of $\mathbb{R}^2$.  Then $\mathbb{R}^2 \setminus \Omega$ is not a polar set if and only if $\Omega$ admits a Green's function.
\end{theorem}

\begin{proof}[Proof of Theorem \ref{theorem 1}]

First suppose, $v(\Omega) = \infty$.  Since $e^{2\lambda(z_0)} > 0$, the inequality part of the Suita conjecture implies that $\Omega$ does not have a Green's function. By Theorem \ref{Myrberg's Result}, $\Omega = \mathbb{C} \setminus P$ where $P$ is a polar set. $P$ is closed because it is the complement in $\mathbb{C}$ of an open set.
\medskip 

Now suppose $v(\Omega) < \infty$.  Since polar sets have two-dimensional Lebesgue measure 0, by Theorem \ref{Myrberg's Result}, $\Omega$ admits a Green's function. As in the proof of the $C^{\infty}$ boundary case, after a translation, $z_0 = 0$ and $K_{\Omega}(\cdot, 0) \equiv v(\Omega)^{-1}$.  Let $\Omega_{\tau} = \{z \in \Omega: g(z, 0) < \tau \}$. Let 
$$
 r_0 := e^{\tau - \lambda(0) - \epsilon}, \quad r_1 := e^{\tau  - \lambda(0) + \epsilon}.
$$
Then for $\tau < 0$ sufficiently negative,
$$
D(0, r_0) \subset \Omega_{\tau} \subset D(0, r_1)
$$  
(cf. \cite{BL16, B}.) Hence,
\begin{eqnarray*}
{e^{-2\epsilon}e^{2\lambda(0)} \over \pi} \leq {e^{2\tau} \over v(\Omega_{\tau})} \leq {e^{2\epsilon}e^{2\lambda(0)} \over \pi}.
\end{eqnarray*}
Letting $\epsilon \to 0^+$, 
$$
 {e^{2 \tau} \over v(\Omega_{\tau})} \approx {e^{2\lambda(0)} \over \pi}, \quad \hbox{ as } \tau \to -\infty.
$$  
By Theorem 3 of \cite{BZ15}, ${e^{2\tau} \over v(\Omega_{\tau})}$  is a decreasing function on $(-\infty, 0]$; hence,
$$
K(0, 0) = {1 \over v(\Omega)} \leq {e^{2\lambda(0)} \over \pi} \leq K(0, 0).
$$
\par
By the equality part of the Suita conjecture, there exists a biholomorphic map $f: D(0, 1)\setminus P \to \Omega$ where $P$ is a closed polar set.  After a M\"obius transformation of the unit disk, we may assume $0 \not\in P$, $f(0) = 0$ and $f'(0) > 0$. Since $P$ is removable for functions in $A^2(D(0, 1) \setminus P)$, $K_{D(0, 1) \setminus P}(\cdot, \cdot) = K_{D(0, 1)}(\cdot, \cdot)$ when both sides are well-defined.  As in the case where $\Omega$ has $C^{\infty}$-boundary, by the transformation law of the Bergman kernel, $f$ is linear.  Hence $\Omega = D(0, f'(0))\setminus f(P)$ and $f(P)$ is a closed polar set.

\end{proof}
\begin{proof}[Proof of Corollary \ref{Corollary}] Since $1 = \int_{\Omega} K(w, z_0) dv(w)$, it follows that
$C = v(\Omega)^{-1} \hbox { and } K(z_0, z_0) = v(\Omega)^{-1}.
$
The result now follows from Theorem 1.

\end{proof}

\section{Proof of Theorem \ref{theorem 2}}

In this section, we provide a short proof of Theorem \ref{theorem 2}. Historically, biholomorphic mappings between Reinhardt domains have been an active area of research since the earliest days of several complex variables (see \cite{R, T}).

\begin{proof}[Proof of Theorem \ref{theorem 2}] It is easy to see that $\Omega$ is complete Reinhardt with algebraic boundary.  To verify that $\Omega$ is strongly convex, one lets $\rho(z) = |z_1|^4 + |z_1|^2 + |z_2|^2 - 1$ and verifies that the real-Hessian of $\mathcal{H}(\rho(z))$ satisfies
$$
w^{\tau}\mathcal{H}(\rho(z_0))w > 0, \quad z_0 \in \partial \Omega, \quad w \in \mathbb{R}^4\setminus\{0\}.
$$
Suppose towards a contradiction that there exists an $F: \mathbb{B}^2 \to \Omega$ which is biholomorphic.  Since the holomorphic automorphism group of $\mathbb{B}^2$ is transitive, we may suppose that $0 \mapsto 0$.  By Henri Cartan's theorem, \cite[Theorem 2.1.3.]{Rud80}, $F$ is linear; that is $F(z) = (a_1z_1 + a_2z_2, a_3z_1 + a_4z_2)$.  Consequently, $F:b \mathbb{B}^2 \to b\Omega$. After composing with a holomorphic rotation of $\mathbb{B}^2$, we may also suppose $F((0, 1)) = (0, 1)$.  Then, $a_2 = 0, a_4 = 1$.  Since for all $\theta \in [0, 2\pi]$,
$$
b\Omega \ni F({1 \over \sqrt{2}}, {1 \over \sqrt{2}}e^{i\theta}) = ({a_1 \over \sqrt{2}}, {a_3 \over \sqrt{2}} + {e^{i\theta} \over \sqrt{2}}),
$$
we see that
$$
{|a_1|^4 \over 4} + {|a_1|^2 \over 2} + {|a_3|^2 \over 2} + Re\langle a_3, e^{i\theta} \rangle + {1 \over 2} = 1,
$$
which implies that $a_3 = 0$. Thus,
\begin{equation} \label{eqn1}
|a_1|^4 + 2|a_1|^2 = 2. 
\end{equation}
Since $(a_1, 0) = F((1, 0)) \in b\Omega$,
\begin{equation} \label{eqn2}
|a_1|^4 + |a_1|^2 = 1.
\end{equation}
Equations \eqref{eqn1} and \eqref{eqn2} do not have a simultaneous solution.  Thus, $F$ does not exist.
 
\end{proof}

\subsection*{Acknowledgements} The authors sincerely thank L\'aszl\'o Lempert, Song-Ying Li, and the anonymous referee for the valuable comments on this paper.

\fontsize{11}{11}\selectfont
 \bigskip
 
\noindent Department of Mathematics, University of California, Irvine, CA 92697-3875, USA
\bigskip
 
\noindent Emails: \quad 1987xindong@tongji.edu.cn  \quad jtreuer@uci.edu
 
\end{document}